\documentclass[12pt]{amsart}
\usepackage{amssymb,amsfonts,latexsym,amscd}

\usepackage[all]{xy}
\usepackage{verbatim}

\newtheorem{theorem}{Theorem}[section]
\newtheorem{lemma}[theorem]{Lemma}
\newtheorem{proposition}[theorem]{Proposition}
\newtheorem{corollary}[theorem]{Corollary}
\theoremstyle{definition}
\newtheorem{definition}[theorem]{Definition}

\newtheorem{example}[theorem]{Example}
\newtheorem{question}[theorem]{Question}

\newtheorem{remark}[theorem]{Remark}

\newcommand{\End}{\text{End}}

\newcommand{\SuperVect}{\text{\rm SuperVect}}

\newcommand{\Hom}{\text{Hom}}

\newcommand{\Rep}{\text{Rep}}

\newcommand{\Spec}{\text{Spec}}
\newcommand{\Corep}{\text{Corep}}
\newcommand{\Vect}{\text{Vect}}

\newcommand{\ot}{\otimes}

\newcommand{\ben}{\begin{enumerate}}
\newcommand{\een}{\end{enumerate}}

\newcommand{\C}{{\mathcal C}}

\begin{document}

\title[Virtually indecomposable tensor categories]
{Virtually indecomposable tensor categories}

\author{Shlomo Gelaki}
\address{Department of Mathematics, Technion-Israel Institute of
Technology, Haifa 32000, Israel} \email{gelaki@math.technion.ac.il}

\date{\today}


\begin{abstract}
Let $k$ be any field. J-P. Serre proved that the spectrum of the
Grothendieck ring of the $k-$representation category of a group is
connected, and that the same holds in {\em characteristic zero} for
the representation category of a Lie algebra over $k$ \cite{se}. We
say that a tensor category $\C$ over $k$ is \emph{virtually
indecomposable} if its Grothendieck ring contains no nontrivial
central idempotents. We prove that the following tensor categories
are virtually indecomposable: Tensor categories with the Chevalley
property; representation categories of affine group schemes;
representation categories of formal groups; representation
categories of affine supergroup schemes (in characteristic $\ne 2$);
representation categories of formal supergroups (in characteristic
$\ne 2$); symmetric tensor categories of exponential growth in
characteristic zero. In particular, we obtain an alternative proof
to Serre's Theorem, deduce that the representation category of any
Lie algebra over $k$ is virtually indecomposable also in {\em
positive characteristic} (this answers a question of Serre
\cite{se}), and (using a theorem of Deligne \cite{d} in the super
case, and a theorem of Deligne-Milne \cite{dm} in the even case)
deduce that any (super)Tannakian category is virtually
indecomposable (this answers another question of Serre \cite{se}).
\end{abstract}

\maketitle

\section{introduction}

The following theorem is due to J-P. Serre.

\begin{theorem} \cite[Corollary 5.5 \& Section 5.1.2; Ex. 3]{se}\label{serre}
Let $k$ be a field.

1) Let $G$ be any group, let $\mathcal{C}:=\Rep(G)$ be the category
of finite-dimensional representations of $G$ over $k$, and let
$Gr(G)$ be its (commutative) Grothendieck ring. Then the spectrum
$\Spec(Gr(G))$ of $Gr(G)$ is connected.

2) Assume that $k$ has {\em characteristic zero}. Let $\mathfrak{g}$
be a Lie algebra over $k$, let $\C:=\Rep(\mathfrak{g})$ be the
category of finite-dimensional representations of $\mathfrak{g}$
over $k$, and let $Gr(\mathfrak{g})$ be its (commutative)
Grothendieck ring. Then $\Spec(Gr(\mathfrak{g}))$ is connected.
\end{theorem}

The proof of Theorem \ref{serre} uses, among other things, the fact
that the semisimple representations of a group $G$ are detected by
their characters, in characteristic zero, and by their Brauer
characters, in positive characteristic.

Recall that the category $\Rep(G)$ is an example of a Tannakian
category \cite{dm} (see Section 2). Motivated by Theorem \ref{serre}
and this fact, Serre asked the following question.

\begin{question} \cite[Section 5.1.2; Ex. 4]{se}\label{serreq}
Let $\C$ be a Tannakian category over any field $k$. Is it true that
$\Spec(Gr(\mathcal{C}))$ is connected? In particular, let
$\mathfrak{g}$ be a Lie algebra over any field $k$ and let
$\C:=\Rep(\mathfrak{g})$ be the category of finite-dimensional
representations of $\mathfrak{g}$ over $k$. Is it true that
$\Spec(Gr(\mathfrak{g}))$ is connected?
\end{question}

Question \ref{serreq} can be extended to any {\em tensor category}
over $k$, namely to a $k-$linear locally finite abelian category
with finite-dimensional $\Hom-$spaces, equipped with an associative
tensor product and unit. (See e.g. \cite{e}, for the definition of a
tensor category and its general theory.)

\begin{definition}
Let $k$ be any field, and let $\C$ be any tensor category over $k$.
Let $R$ be any commutative ring. We say that $\C$ is {\em virtually
indecomposable over $R$} if its Grothendieck ring
$R\otimes_\mathbb{Z} Gr(\mathcal{C})$ with $R-$coefficients has no
nontrivial central idempotents, and that $\C$ is {\em strongly
virtually indecomposable over $R$} if $R\otimes_\mathbb{Z}
Gr(\mathcal{C})$ has no nontrivial idempotents. In the case
$R=\mathbb{Z}$ we shall suppress the phrase ``over $\mathbb{Z}$".
\end{definition}

\begin{question}\label{q}
Is it true that any tensor category over any field is virtually
indecomposable? Strongly virtually indecomposable?
\end{question}

Our goal in this paper is to provide a positive answer to Question
\ref{q} for a variety of tensor categories over any field $k$. More
precisely, we prove that the following tensor categories are
virtually indecomposable:
\begin{itemize}
\item
Tensor categories with the Chevalley property.
\item
Representation categories of affine group schemes.
\item
Representation categories of formal groups.
\item
Representation categories of affine supergroup schemes (in
characteristic $\ne 2$).
\item
Representation categories of formal supergroups (in characteristic
$\ne 2$).
\item
Symmetric tensor categories of exponential growth in characteristic
zero.
\end{itemize}
In particular, we obtain both an alternative proof to Theorem
\ref{serre} and a positive answer to Question \ref{serreq}.

\par

{\bf Acknowledgments.} The author is grateful to J-P. Serre for
sending him his proof of Theorem \ref{serre}, for suggesting
Question \ref{serreq}, and for helpful comments.

The author is indebted to P. Etingof for his help with the proofs,
and for his interest in the paper and encouragement.

The author thanks V. Ostrik for telling him about the paper
\cite{r}.

The research was partially supported by The Israel Science
Foundation (grant No. 317/09).

\section{The main results}

The following standard lemma shows that without loss of generality
we may (and shall) work over an algebraically closed field.

\begin{lemma} If $\C$ is a locally finite abelian category over a field $k$
then the map $Gr(\C)\to Gr(\C\otimes _k \overline{k})$ is injective.
\end{lemma}

\begin{proof} It is well known that $\C$ is equivalent to
the category of finite-dimensional $A-$comodules over $k$, where $A$
is a coalgebra over $k$. Let us denote $Gr(\C)$ by $Gr(A)$. We need
to show that the map $Gr(A)\to Gr(A\otimes_k \overline{k})$ is
injective. Clearly, we may assume that $A$ is finite-dimensional, so
$\C=\Rep(A^*)$. Then we can pass to the quotient of $A^*$ by its
radical and assume that $A^*$ is semisimple. So we can assume that
$A^*$ is simple, i.e., $A^*=Mat_n(D)$, $D$ a division algebra over
$k$. But in this case the claim is obvious since $Gr(A)=\mathbb{Z}$.
\end{proof}

\begin{corollary}\label{algcl} A tensor category $\C$ over $k$ is virtually
indecomposable if $\C\otimes_k
\overline{k}$ is virtually indecomposable. \qed
\end{corollary}

Therefore, throughout the paper we shall work over an
\textbf{algebraically closed field $k$}.

\subsection{Based rings.} In Section 3.1 we recall the definition of a
unital based ring, and then prove in Section 3.2 the following
theorem about them.

\begin{theorem}\label{based}
Let $A$ be any unital based ring. Then $A$ is virtually
indecomposable.
\end{theorem}

Recall that a $k-$linear abelian rigid tensor category $\C$ is said
to have the {\em Chevalley property} if the tensor product of any
two semisimple objects of $\C$ is also semisimple. In other words,
the subcategory $\C_{ss}$ of semisimple objects in $\C$ is a tensor
subcategory. For example, in characteristic zero, $\C=\Rep(G)$ and
$\C=\Rep(\mathfrak{g})$, where $G$ is any group and $\frak{g}$ is
any Lie algebra, have the Chevalley property \cite{c}. Of course, if
$\C$ is semisimple (e.g., a fusion category) then $\C$ has the
Chevalley property.

Now, if  $\C$ has the Chevalley property then $Gr(\C)=Gr(\C_{ss})$,
so $Gr(\C)$ is a unital based ring. Hence, Theorem \ref{based}
implies the following corollary.

\begin{corollary}\label{based1}
Let $\C$ be a $k-$linear abelian rigid tensor category. If $\C$ has
the Chevalley property then $\C$ is virtually indecomposable. \qed
\end{corollary}

\begin{remark}
In general it is not true that the representation categories of
groups and Lie algebras in positive characteristic have the
Chevalley property, and likewise for supergroups and Lie
superalgebras in any characteristic.
\end{remark}

\subsection{The Hopf algebra case.} In Section 4 we prove
the following innocent looking result, which will turn out to play
the key role in proving our results concerning (super)groups and
(super)Lie algebras.

\begin{theorem}\label{mainhopf}
Let $H$ be a (not necessarily commutative) Hopf algebra over a field
$k$, and let $\Corep(H)$ denote the tensor category of
finite-dimensional $H-$comodules over $k$. Suppose that $I$ is a
Hopf ideal in $H$ such that $\bigcap_{n\ge 1}I^n=0$. Let $R$ be any
commutative ring and, if the characteristic of $k$ is $p>0$, assume
that $\bigcap_{n\ge 1}p^nR=0$. Then, if $\Corep(H/I)$ is virtually
indecomposable over $R$ then so is $\Corep(H)$.
\end{theorem}

\begin{remark}\label{lazy}
In fact, Theorem \ref{mainhopf} holds also, with the same proof, in
the topological case (i.e., when $H$ is a topological Hopf algebra;
see below).
\end{remark}

\subsection{The group case.} In Section 5 we use Theorem
\ref{mainhopf} to prove increasingly strong results, culminating in
the following theorem.

\begin{theorem}\label{maingroup}
Let $k$ be any field, and let $G$ be an affine group scheme over
$k$. Let $S$ be the set of all primes not equal to the
characteristic of $k$ and not dividing $|G/G^0|$. Then
$\Spec(\mathbb{Z}[S^{-1}]\otimes _\mathbb{Z} Gr(G))$ is connected.
\end{theorem}

Theorem \ref{maingroup} generalizes to formal groups. Recall that a
{\em formal group} $G$ over a field $k$, whose subset of closed
points (= reduced part) is the affine proalgebraic group
$\overline{G}$ over $k$, is the following algebraic structure. We
have a structure algebra $\mathcal{O}(G)$ over $k$, which has an
ideal $I$ such that $\mathcal{O}(G)/I=\mathcal{O}(\overline{G})$,
and $\mathcal{O}(G)$ is complete and separated in the topology
defined by $I$ (i.e., $\mathcal{O}(G)=\underleftarrow{\lim}
\mathcal{O}(G)/I^m$). Finally, we have a cocommutative coproduct
$\Delta: \mathcal{O}(G)\to \mathcal{O}(G)\widehat{\otimes
}\mathcal{O}(G)$, where the latter completed tensor product is
$\underleftarrow{\lim}(\mathcal{O}(G)/I^m\otimes
\mathcal{O}(G)/I^m)$, defining a topological Hopf algebra structure
on $\mathcal{O}(G)$, such that $I$ is a Hopf ideal, and the
isomorphism $\mathcal{O}(G)/I\to \mathcal{O}(\overline{G})$ is a
Hopf algebra isomorphism.

Thus, combining Theorems \ref{mainhopf} and \ref{maingroup}, we
obtain the following result.

\begin{theorem}\label{mainformal}
Let $k$ be any field, and let $G$ be a formal group over $k$ with
reduced part $\overline{G}$. Let $S$ be the set of all primes not
equal to the characteristic of $k$ and not dividing
$|\overline{G}/\overline{G}^0|$. Then
$\Spec(\mathbb{Z}[S^{-1}]\otimes _\mathbb{Z} Gr(G))$ is connected.
\qed
\end{theorem}

Therefore, as an immediate corollary of Theorem \ref{mainformal}
(the case $\overline{G}=1$), we deduce a positive answer to the
second part of Serre's Question \ref{serreq}. Nevertheless, in
Section 4.3 we shall also give a self contained proof of this
theorem in the positive characteristic case.

\begin{theorem}\label{lie}
Let $\mathfrak{g}$ be a Lie algebra over any field $k$ and let
$\C:=\Rep(\mathfrak{g})$ be the category of finite-dimensional
representations of $\mathfrak{g}$ over $k$. Then
$\Spec(Gr(\mathfrak{g}))$ is connected. \qed
\end{theorem}

\begin{remark}
Note that the case $\overline{G}=1$ (formal groups with one closed
point) reduces to Lie algebras in characteristic zero, but in
positive characteristic it contains much more.
\end{remark}

Recall that a Hopf algebra $H$ over a field $k$ is called {\em
coconnected} if every simple $H-$comodule over $k$ is trivial (see
e.g. \cite{eg} where, in particular, coconnected Hopf algebras over
$\mathbb{C}$ are classified in Theorem 4.2). We have the following
result which extends Corollary \ref{lie}.

\begin{theorem}
Let $H$ be a coconnected Hopf algebra over any field $k$, and let
$S$ be the set of all primes not equal to the characteristic of $k$.
Then $\Rep(H)$ is virtually indecomposable over
$\mathbb{Z}[S^{-1}]$. \qed
\end{theorem}

\begin{proof}
If $H$ is coconnected then $H^*$ is a topological Hopf algebra with
maximal ideal $I:=Ker(\epsilon)$, which is complete and separated in
the topology defined by $I$ (as the powers of $I$ are orthogonal to
the terms of the coradical filtration of $H$). So the claim follows
from the topological version of Theorem \ref{mainhopf} (see Remark
\ref{lazy}).
\end{proof}

\subsection{The supergroup case.} In Section 6.1 we recall the
notion of a Hopf superalgebra, and in Section 6.2 we recall the
notions of an affine supergroup scheme and a formal supergroup over
$k$. We then generalize in Section 6.3 the results from Section 5 to
the super-case (assuming the characteristic of $k\ne 2$).

Let $\mathcal{G}$ be an affine supergroup scheme or, more generally,
a formal supergroup, and let $u\in \mathcal{G}$ be an element of
order $2$ acting by parity on the algebra of regular functions
$\mathcal{O}(\mathcal{G})$. Let $\Rep(\mathcal{G},u)$ be the
category of representations of $\mathcal{G}$ on finite-dimensional
supervector spaces over $k$ on which $u$ acts by parity, and let
$Gr(\mathcal{G},u)$ be its Grothendieck ring.

\begin{theorem}\label{mainformalsuper}
Let $k$ be any field of characteristic $\ne 2$. Let $\mathcal{G}$ be
an affine supergroup scheme over $k$ or, more generally, a formal
supergroup over $k$. Let $S$ be the set of all primes $\ne 2$ not
equal to the characteristic of $k$ and not dividing
$|\mathcal{G}/\mathcal{G}^0|$. Then $\Spec(\mathbb{Z}[S^{-1}]\otimes
_\mathbb{Z} Gr(\mathcal{G},u))$ is connected.
\end{theorem}

\begin{remark}
Note that the prime $2$ must be excluded (i.e., cannot be inverted).
Indeed, already in the category $\SuperVect$ of finite-dimensional
supervector spaces over $k$ (see Section 6), the element
$\frac{1}{2}(k_0\oplus k_1)$ is a nontrivial idempotent.
\end{remark}

Recall that a {\em Lie superalgebra} over a field $k$ is a Lie
algebra in $\SuperVect$ (see e.g, \cite{b}). In other words, a Lie
superalgebra $\mathfrak{g}=\mathfrak{g}_0\oplus \mathfrak{g}_1$ is a
supervector space over $k$, equipped with an operation
$[\,,\,]:\mathfrak{g}\otimes \mathfrak{g}\to \mathfrak{g}$
satisfying the following axioms: $[x,y] = -(-1)^{|x||y|}[y,x]$ and
$[x,[y,z]]=[[x,y],z] + (-1)^{|x||y|}[y,[x,z]]$, for homogeneous
elements $x,y\in \mathfrak{g}$ and $z\in \mathfrak{g}$. The
following result on Lie superalgebras is an immediate corollary of
Theorem \ref{mainformalsuper}.

\begin{corollary}\label{superlie}
Let $\mathfrak{g}$ be a Lie superalgebra over a field $k$ of
characteristic $\ne 2$. Let $S$ be the set of all primes $\ne 2$ not
equal to the characteristic of $k$. Then
$\Spec(\mathbb{Z}[S^{-1}]\otimes _\mathbb{Z} Gr(\mathfrak{g}))$ is
connected. \qed
\end{corollary}

By a theorem of Deligne \cite{d} in characteristic zero, the
categories $\Rep(\mathcal{G},u)$ exhaust all $k-$linear abelian
symmetric rigid tensor categories of exponential growth. Hence, we
deduce the following corollary.

\begin{corollary}\label{maintanak}
If $\C$ is a $k-$linear abelian symmetric rigid tensor category of
exponential growth over an algebraically closed field $k$ of
characteristic zero, then $\C$ is virtually indecomposable. \qed
\end{corollary}

Recall that a {\em (super)Tannakian} category over a field $k$ is a
$k-$linear abelian symmetric rigid tensor category $\C$, with
$\End({\bf 1})=k$, where ${\bf 1}$ denotes the unit object, which
admits a fiber functor to the category of finite-dimensional
(super)vector spaces (see \cite{d}). In the following proposition we
deduce a positive answer to the first part of Serre's Question
\ref{serreq}.

\begin{proposition} \label{tanak}
A (super)Tannakian category $\C$ over any field $k$ is virtually
indecomposable.
\end{proposition}

\begin{proof}
By (the super analog of) a theorem of Deligne-Milne \cite {dm}
(which is in \cite{d}), $\C$ is equivalent to a category of the form
$\Rep(\mathcal{G},u)$, so the claim follows by Theorem
\ref{mainformalsuper}.
\end{proof}

\section{The virtually indecomposability of a unital based ring}

In characteristic zero there is an alternative (``combinatorial")
proof of (a slight generalization of) Theorem \ref{serre} in the
framework of unital based rings.

\subsection{Based rings} Let $A$ be a ring with a
distinguished $\mathbb{Z}-$basis $\{b_i\}$, $i\in I$, (not
necessarily of finite rank), which contains the unit element $1$,
such that $b_i b_j = \sum_{k} n_{ij}^k b_k$, where $n_{ij}^k\in
\mathbb{Z}^+$. The bilinear map $(\sum_{i} n_{i} b_i,\sum_{i} m_{i}
b_i)\mapsto \sum_{i}n_i m_i$ defines a positive inner product
$(\,,\,):A\times A\to \mathbb{Z}$ on $A$. We call $A$ a {\em unital
based ring} if there is an involution $i\mapsto i^*$ such that the
induced map $x = \sum_{i} n_{i} b_i \mapsto x^* := \sum_{i} n_{i}
b_{i^*}$ satisfies $(xy,z)=(x, zy^*)=(y,x^*z)$ for all $x,y,z\in A$.
In particular, it follows that the matrix of multiplication by $x^*$
is transposed to the matrix of multiplication by $x$, for any $x\in
A$.

\begin{example}\label{ex}
If $\C$ is a $k-$linear semisimple rigid tensor category, its
Grothendieck ring $Gr(\C)$ is a unital based ring. A typical example
of such category is the category $\C:=\Corep(H)$ of
finite-dimensional comodules of a cosemisimple Hopf algebra $H$. The
distinguished $\mathbb{Z}-$basis of $Gr(\C)$ consists of the
isomorphism classes of simple $H-$comodules, and the involution $*$
is given by taking the $k-$linear dual of a comodule.
\end{example}

\subsection{The proof of Theorem \ref{based}.}
Let $e\ne 1$ be a central idempotent in $A$. We have to show that
$e=0$. We first note that $e$ is a projection operator on an inner
product space, which is normal (i.e $ee^*=e^*e$), so $e$ is
self-adjoint. Indeed,
$(e(1-e^*),e(1-e^*))=(e^*e(1-e^*),1-e^*)=(ee^*(1-e^*),1-e^*)=0$.
Thus by positivity of the inner product, $e(1-e^*)=0$, so $e=ee^*$,
hence $e=e^*$.

Then $e$ is an orthogonal projector to a proper subspace of
$\mathbb{R}\otimes_\mathbb{Z} A$, which does not contain $1$. So
$0\le (e,e)=(e1,e1)<(1,1)=1$. But $(e,e)$ is an integer, so
$(e,e)=0$, and hence $e=0$. \qed

\begin{remark} It is interesting to mention here a
classical result of Kaplansky which asserts that there is no
nontrivial idempotent in the integral group ring of any (not
necessarily commutative) group (see \cite{k}, \cite{p}), i.e., the
integral group ring of any group is strongly virtually
indecomposable. Equivalently, the tensor category $Vec_G$ of
$G-$graded vector spaces over $k$ is strongly virtually
indecomposable for any group $G$.

In fact, Proposition 3 in \cite{r} extends the result of Kaplansky
to fusion rings (= unital based rings of finite rank). Equivalently,
any fusion category is strongly virtually indecomposable.
\end{remark}

\section{The proof of Theorem \ref{mainhopf}.}

In this section we let $H$ be a Hopf algebra (not necessarily
commutative) over $k$, and  $\mathcal{C}:=\Corep(H)$ be the category
of finite-dimensional right comodules of $H$. Then $\C$ is a
$k-$linear abelian rigid tensor category in which every object has a
finite length. Let $Gr(\mathcal{C})$ be the Grothendieck ring of
$\mathcal{C}$; it is the free $\mathbb{Z}-$algebra with a
distinguished basis formed by the classes $[X]$ of the simple
objects $X\in \mathcal{C}$.

\subsection{Characters in Hopf algebras.}
Recall that any $M\in \C$ has a canonical rational $H^*-$module
structure.

\begin{definition}
For an object $M\in\mathcal{C}$, the character $ch(M)$ of $M$ is the
character of the $H^*-$module $M$. In other words, the character
$ch(M)$ is the function $H^*\to k$ defined by
$ch(M)(x):=tr(x_{|M})$.
\end{definition}
Clearly, $ch(M)\in H$, $ch(M)ch(N)=ch(M\otimes N)$ and
$ch(M)+ch(N)=ch(M+N)$. Moreover, if $M_1,\dots,M_n$ are the distinct
composition factors of $M$, with multiplicities $a_1,\dots,a_n$,
then $ch(M)=\sum_{i=1}^n a_ich(M_i)$. In other words, the character
of $M$ and the character of its semisimplification $\oplus_{i=1}^n
a_i M_i$ coincide. We therefore have a well defined $k-$algebra
homomorphism
$$ch:k\otimes_{\mathbb{Z}}Gr(\mathcal{C})\to
H,\,\,a\otimes [M]\mapsto a\cdot ch(M).$$

\begin{proposition}\label{injective}
The character map $ch$ is injective. In other words, if
$M,N\in\mathcal{C}$ with $ch(M)=ch(N)$, then $[M]=[N]$ in
$k\otimes_{\mathbb{Z}}Gr(\mathcal{C})$.
\end{proposition}

\begin{proof}
It is enough to show that if $\sum_{i=1}^{m} a_ich(M_i)=0$ on $H^*$,
for some finite number of non-isomorphic irreducible comodules
$M_i\in \C$ and some elements $a_i\in k$, then $a_i=0$ for all $i$.

Indeed, by the density theorem, the map $H^*\to \oplus_i\End_k(M_i)$
is surjective, so we can choose an element $x\in H^*$ which maps to
$0$ on $\End_k(M_j)$ for $j\ne i$, and to an element with trace $1$
on $\End_k(M_i)$, which implies that $a_i=0$ for all $i$.
\end{proof}

\begin{remark}
Note that if the characteristic of $k$ is zero then Proposition
\ref{injective} implies that the character of $M$ determines the
composition factors of $M$ together with their multiplicities, i.e.,
$ch:Gr(\mathcal{C})\to H$ is injective (so in particular, if $M$,
$N$ are semisimple then $M$, $N$ are isomorphic). On the other hand,
if the characteristic of $k$ is $p>0$ then Proposition
\ref{injective} implies only that the character of $M$ determines
the composition factors of $M$ {\em together with their
multiplicities modulo $p$}.
\end{remark}

\subsection{The proof of Theorem \ref{mainhopf}.}
Set $\overline{\C}:=\Corep(H/I)$. The surjection of Hopf algebras
$H\twoheadrightarrow H/I$ induces a tensor functor $\C\to
\overline{\C}$, which in turn induces a ring homomorphism
$R\otimes_\mathbb{Z} Gr(\C)\to R\otimes_\mathbb{Z}
Gr(\overline{\C})$. Suppose $E\in R\otimes_\mathbb{Z} Gr(\C)$ is an
idempotent which is not $0$ or $1$, and let $e$ be the image of $E$
in $R\otimes_\mathbb{Z} Gr(\overline{\C})$. By assumption, $e$ is
either $0$ or $1$. Without loss of generality we may assume that
$e=0$, replacing $E$ by $1-E$ if needed.

Now, if $k$ has characteristic $p>0$, at least one of the
coefficients of $E$ is not divisible by $p$. Indeed, if $E=pF$ then
$E^n=p^nF^n=E$, so $E\in p^nR\otimes_\mathbb{Z} Gr(\C)$ for all $n$,
and hence it is zero, which is a contradiction. Therefore, the image
$E'$ of $E$ in $R\otimes_\mathbb{Z} k\otimes_{\mathbb{Z}} Gr(\C)$
(which is $(R/pR)\otimes_{\mathbb{F}_p} k\otimes_{\mathbb{Z}}
Gr(\C)$ in positive characteristic) is nonzero, and the image $e'$
of $e$ in $R\otimes_\mathbb{Z} k\otimes_{\mathbb{Z}}
Gr(\overline{\C})$ is zero (as $e=0$).

Now, using the embedding  $ch:k\otimes_{\mathbb{Z}}
Gr(\C)\hookrightarrow H$, we get a nonzero idempotent $ch(E')$ in
$R\otimes_\mathbb{Z} H$, which has zero image in
$R\otimes_\mathbb{Z} H/I$ (this image is $ch(e')$). This implies
that $ch(E')\in R\otimes_\mathbb{Z} I$. But since $ch(E')$ is an
idempotent, $ch(E')^n=ch(E')$ for all $n$, so $ch(E')\in
\bigcap_{n\ge 1}(R\otimes_\mathbb{Z} I^n)=0$, which is a
contradiction. \qed

\subsection{A proof of Theorem \ref{lie}.}
Let $\mathfrak{g}$ be a Lie algebra over a field $k$ of
characteristic $p>0$; we may assume without loss of generality that
$k$ is algebraically closed (see Corollary \ref{algcl}). Let
$A:=U(\mathfrak{g})^*$ be the dual algebra of the universal
enveloping algebra $U(\mathfrak{g})$ of $\mathfrak{g}$ (it is a
topological Hopf algebra in the topology defined by the maximal
ideal $I$ of $A$). Let $E\in Gr(\mathfrak{g})$ be an idempotent
which is not $0$ or $1$. We can assume that $Tr_E(1)=0$ modulo $p$
by replacing $E$ with $1-E$ if needed.

Now, at least one of the coefficients of $E$ is not divisible by
$p$. Indeed, otherwise $(E/p)^n=E/p^n$, so $E/p^n\in
Gr(\mathfrak{g})$ for all $n$, but $E/p^n$ does not have integer
coefficients for large enough $n$. Consequently, the image $E'$ of
$E$ in $k\otimes_\mathbb{Z} Gr(\mathfrak{g})$ is nonzero. Hence,
using Proposition \ref{injective}, we get a nonzero idempotent
$ch(E')$ in $A$. On the other hand, the augmentation map $A\to k$
maps $ch(E')$ to zero (since $Tr_E(1)=0$ modulo $p$). So $ch(E')$ is
contained in $I$. But $ch(E')$ is an idempotent, so it is contained
in any power $I^n$ of $I$. But $\bigcap_{n\ge 1} I^n=0$, so $ch(E')$
is zero, which is a contradiction. \qed

\section{The proof of Theorem \ref{maingroup}}

The proof of Theorem \ref{maingroup} will be carried in several
steps.

\subsection{$G$ is a reductive abelian affine algebraic group over
$k$.} Recall that if $G$ is a reductive abelian affine algebraic
group over $k$ then $G\cong G^0\times A$, where $G^0=\mathbb{G}_m^n$
is the $n-$dimensional torus over $k$ and $A$ is a finite abelian
group of order prime to $p$ (in case the characteristic of $k$ is
$p>0$) (see e.g., \cite{sp}). In particular, all finite-dimensional
simple representations of $G$ over $k$ are $1-$dimensional, and
their isomorphism types are parameterized by pairs $(z,\chi)$, where
$z\in \mathbb{Z}^n$ and $\chi\in \widehat{A}$. Thus all idempotents
in $\mathbb{Q}\otimes_{\mathbb{Z}} Gr(G)$ can be easily described in
this case (they involve a factor of $1/|A|$), so the result follows
in a straightforward manner.

\subsection{$G$ is any abelian affine algebraic group over $k$.}
Recall that if $G$ is an abelian affine algebraic group over $k$
then $G\cong G_s\times G_u$, where $G_s$ and $G_u$ are the subgroups
of semisimple and unipotent elements of $G$, respectively (see e.g.,
\cite{sp}). So $\Rep(G)$ and $\Rep(G_s)$ have the same Grothendieck
rings, and the claim follows from 5.1.

\subsection{$G$ is any affine algebraic group over $k$.}
Let $G$ be any affine algebraic group over $k$, and let
$\mathcal{O}(G)$ be its coordinate Hopf algebra; it is a finitely
generated commutative reduced Hopf algebra over $k$. Suppose $e\in
Gr(\Rep(G))=Gr(\Corep(\mathcal{O}(G)))$ is an idempotent. Then
$ch(e)\in \mathcal{O}(G)$ is an idempotent, so in particular a class
function on $G$ taking the values $0,1$ (here we are using the trick
with dividing by $p$, as we did in the proof of Theorem
\ref{mainhopf}, which explains why we cannot invert $p$). Therefore
it suffices to prove that $ch(e)(g)=ch(e)(1)$ for all $g\in G$. But
for that purpose we may assume that $G$ is the (Zariski closure of
the) cyclic group with generator $g$, so $G$ is abelian and the
claim follows from 5.2.

\subsection{$G$ is an affine proalgebraic group over $k$.}
Let $G$ be an affine proalgebraic group over $k$, and let
$\mathcal{O}(G)$ be its coordinate Hopf algebra over $k$; it is a
commutative reduced Hopf algebra over $k$ (not necessarily finitely
generated). But it is well known that $\mathcal{O}(G)$ is the
inductive limit of its finitely generated Hopf subalgebras (see e.g.
\cite{a}), so the claim follows in a straightforward manner from
Part 5.3.

\subsection{$G$ is an affine group scheme over $k$.}
Let $G$ be an affine group scheme, and let $\mathcal{O}(G)$ be the
commutative Hopf algebra representing the group functor $G$ (it is
not necessarily reduced) (see e.g., \cite{w}). Since $k$ is
algebraically closed, it is well known (essentially by Hilbert
Nullstellensatz) that the nilradical $I$ of $\mathcal{O}(G)$ is a
Hopf ideal. Now, the commutative reduced Hopf algebra
$\mathcal{O}(G)/I$ over $k$ represents an affine proalgebraic group
over $k$, so by 5.4, $\Corep(\mathcal{O}(G)/I)$ is virtually
indecomposable. Finally, since $\bigcap_{n\ge 1}I^n=0$, it follows
from Theorem \ref{mainhopf} that $\Corep(\mathcal{O}(G))=\Rep(G)$ is
virtually indecomposable, as claimed.

This concludes the proof of Theorem \ref{maingroup}. \qed

\section{The proof of Theorem \ref{mainformalsuper}}

\subsection{Hopf superalgebras.} Recall that a {\em Hopf superalgebra} over
a field $k$ is a Hopf algebra in the $k-$linear abelian symmetric
tensor category $\SuperVect$ of supervector spaces over $k$ (see
e.g, \cite{b}). In other words, a Hopf superalgebra
$\mathcal{H}=\mathcal{H}_0\oplus \mathcal{H}_1$ is an ordinary
$\mathbb{Z}_2-$graded associative unital algebra over $k$ (i.e., a
{\em superalgebra}), equipped with a coassociative morphism
$\Delta:\mathcal{H}\to \mathcal{H}\ot \mathcal{H}$ in $\SuperVect$,
which is multiplicative in the super-sense, and with a counit and
antipode satisfying the standard axioms. Here multiplicativity in
the super-sense means that $\Delta$ satisfies the relation
\begin{equation*}\label{sd}
\Delta(ab)=\sum (-1)^{|a_2||b_1|}a_1b_1\ot a_2b_2,
\end{equation*}
where $a,b\in \mathcal{H}$ are homogeneous elements, $|a|$ and $|b|$
denote the degrees of $a$ and $b$, $\Delta(a)=\sum a_1\ot a_2$ and
$\Delta(b)=\sum b_1\ot b_2$.

A Hopf superalgebra $\mathcal{H}$ is said to be {\em commutative} if
$ab=(-1)^{|a||b|}ba$ for all homogeneous elements $a,b\in
\mathcal{H}$.

Let $J(\mathcal{H}):=(\mathcal{H}_1)=\mathcal{H}_1^2\oplus
\mathcal{H}_1$ be the Hopf ideal of $\mathcal{H}$ generated by the
odd elements $\mathcal{H}_1$. Then the quotient
$\overline{\mathcal{H}}:=\mathcal{H}/J(\mathcal{H})$ is an ordinary
Hopf algebra. Note that if $\mathcal{H}$ is commutative then
$J(\mathcal{H})$ consists of nilpotent elements.

Let us recall the following useful construction, introduced in
Section 3.1 in \cite{aeg}. Let $\mathcal{H}$ be any Hopf
superalgebra over $k$, and let
$\widetilde{\mathcal{H}}:=k[\mathbb{Z}_2]\ltimes \mathcal{H}$ be the
semidirect product Hopf superalgebra with comultiplication
$\widetilde{\Delta}$ and antipode $\widetilde{S}$, where the
generator $g$ of $\mathbb{Z}_2$ acts on $\mathcal{H}$ by
$ghg^{-1}=(-1)^{|h|}h$. For $x\in \widetilde{\mathcal{H}}$, write
$\widetilde{\Delta}(x)=\widetilde{\Delta}_0(x)+\widetilde{\Delta}_1(x)$,
where $\widetilde{\Delta}_0(x)\in \widetilde{\mathcal{H}}\ot
\widetilde{\mathcal{H}}_0$ and $\widetilde{\Delta}_1(x)\in
\widetilde{\mathcal{H}}\ot \widetilde{\mathcal{H}}_1$. Then one can
define an \emph{ordinary} Hopf algebra structure on the algebra
$\mathcal{H}':=\widetilde{\mathcal{H}}$, with comultiplication and
antipode maps given by
$\Delta(h):=\widetilde{\Delta}_0(h)-(-1)^{|h|}( g\ot
1)\widetilde{\Delta}_1(h)$ and  $S(h):=g^{|h|}\widetilde {S}(h)$,
$h\in \mathcal{H}'$ (see Theorem 3.1.1 in \cite{aeg}).

\begin{theorem}\label{supord}
Let $\mathcal{H}$ be a Hopf superalgebra over a field $k$ of
characteristic $\ne 2$, and let $\Corep(\mathcal{H})$ be the tensor
category of finite-dimensional $\mathcal{H}-$comodules in
$\SuperVect$ over $k$. The tensor categories $\Corep(\mathcal{H})$
and $\Corep(\mathcal{H}')$ are equivalent.
\end{theorem}

\begin{proof}  An
$\mathcal{H}-$comodule in $\SuperVect$ is a (continuous)
$\mathbb{Z}_2-$graded $\mathcal{H}^*-$module in the category $\Vect$
of vector spaces over $k$, which is the same as a continuous module
over the algebra $k[\mathbb{Z}_2]\ltimes \mathcal{H}^*$. But
$k[\mathbb{Z}_2]\ltimes \mathcal{H}^*$ is isomorphic to
$(\mathcal{H}')^*$ as an algebra, so a $\mathcal{H}-$comodule in
$\SuperVect$ is the same thing as a $\mathcal{H}'-$comodule in
$\Vect$. Finally, it is easy to check that this equivalence is in
fact a tensor equivalence.
\end{proof}

We note that as a consequence of Theorem \ref{supord}, we can define
the character map
$ch:k\otimes_{\mathbb{Z}}Gr(\Corep(\mathcal{H}))\to
\mathcal{H}',\,\,a\otimes [M]\mapsto a\cdot ch(M)$, and deduce from
Proposition \ref{injective} that it is an injective $k-$algebra
homomorphism.

\subsection{Affine supergroup schemes and formal supergroups.}
Recall that an {\em affine supergroup scheme} $\mathcal{G}$ is the
spectrum of a (not necessarily finitely generated) commutative Hopf
superalgebra $\mathcal{O}(\mathcal{G})$ over $k$ (see e.g.,
\cite{d}). In other words, it is a functor $\mathcal{G}$ from the
category of supercommutative algebras to the category of groups
defined by $A\mapsto
\mathcal{G}(A):=\Hom(\mathcal{O}(\mathcal{G}),A)$, where
$\Hom(\mathcal{O}(\mathcal{G}),A)$ is the group of algebra maps
$\mathcal{O}(\mathcal{G})\to A$ in $\SuperVect$.

Recall that a \emph{formal supergroup} $\mathcal{G}$ over a field
$k$, with reduced part $G$, is the following algebraic structure. We
have a superalgebra $\mathcal{O}(\mathcal{G})$ over $k$, which has
an ideal $\mathcal{I}$ such that $\mathcal{O}(\mathcal{G})$ is
complete and separated in the topology defined by $\mathcal{I}$
(i.e., $\mathcal{O}(\mathcal{G})=\underleftarrow{\lim}
\mathcal{O}(\mathcal{G})/\mathcal{I}^m$), and
$\mathcal{O}(\mathcal{G})/\mathcal{I}=\mathcal{O}(G)$. Finally, we
have a supercocommutative coproduct $\Delta:
\mathcal{O}(\mathcal{G})\to \mathcal{O}(\mathcal{G})\widehat{\otimes
}\mathcal{O}(\mathcal{G})$, where the latter completed tensor
product is
$\underleftarrow{\lim}(\mathcal{O}(\mathcal{G})/\mathcal{I}^m\otimes
\mathcal{O}(\mathcal{G})/\mathcal{I}^m)$, defining a topological
Hopf algebra structure on $\mathcal{O}(\mathcal{G})$, such that
$\mathcal{I}$ is a Hopf ideal, and the isomorphism
$\mathcal{O}(\mathcal{G})/\mathcal{I}\to \mathcal{O}(G)$ is a Hopf
superalgebra isomorphism.

Let $\mathcal{G}$ be an affine supergroup scheme. The ordinary
commutative Hopf algebra $\overline{\mathcal{O}(\mathcal{G})}$ is
isomorphic to $\mathcal{O}(G)$ for some affine group scheme $G$,
which is referred to as the {\em even part} of $\mathcal{G}$.

Let $\mathcal{G}$ be an affine supergroup scheme over $k$, or, more
generally, a formal supergroup over $k$, and let $\Rep(\mathcal{G})$
denote the category of finite-dimensional algebraic representations
of $\mathcal{G}$ in $\SuperVect$ over $k$. Then $\Rep(\mathcal{G})$
is a $k-$linear abelian symmetric rigid tensor category with
$\End({\bf 1})=k$, where ${\bf 1}$ denotes the unit object, which
admits a fiber functor (= a symmetric tensor functor) to the full
$k-$linear abelian symmetric rigid tensor subcategory of
$\SuperVect$ whose objects are the finite-dimensional supervector
spaces. Just like in the even case, $\Rep(\mathcal{G})$ is
equivalent to $\Corep(\mathcal{O}(\mathcal{G}))$.

\subsection{The proof of Theorem \ref{mainformalsuper}.} Let
$\mathcal{G}$ be an affine supergroup scheme over $k$, let
$\mathcal{H}:=\mathcal{O}(\mathcal{G})$ be the commutative Hopf
superalgebra representing the group functor $\mathcal{G}$, and let
$\mathcal{H}'$ be the ordinary Hopf algebra associated with
$\mathcal{H}$. Then the quotient
$\mathcal{H}'/(k[\mathbb{Z}_2]\ltimes J(\mathcal{H}))$ is a
commutative Hopf algebra representing the group functor
$\mathbb{Z}_2\times G$, where $G$ is the even part of $\mathcal{G}$.
Therefore $\Corep(\mathcal{H}'/(k[\mathbb{Z}_2]\ltimes
J(\mathcal{H})))$ is virtually indecomposable by Theorem
\ref{maingroup}, and it follows from Theorems \ref{mainhopf} and
\ref{supord} that $\Corep(\mathcal{H})$ is virtually indecomposable,
as claimed.

For formal supergroups $\mathcal{G}$ over $k$ the proof is
completely parallel using Theorem \ref{mainformal} about formal
groups over $k$. \qed

\end{document}